\documentclass[a4paper,reqno]{amsart}

\usepackage{amssymb}
\usepackage{amstext}
\usepackage{amsmath}
\usepackage{amscd}
\usepackage{amsthm}
\usepackage{amsfonts}
\usepackage{enumerate}
\usepackage{graphicx}
\usepackage{latexsym}
\usepackage{mathrsfs}
\usepackage{mathtools}
\usepackage[all]{xy}
\xyoption{all}

\usepackage{pstricks}
\usepackage{lscape}
\usepackage{comment}

\newtheorem{theorem}{Theorem}[section]

\newtheorem{corollary}[theorem]{Corollary}
\newtheorem{lemma}[theorem]{Lemma}
\newtheorem{proposition}[theorem]{Proposition}
\newtheorem{definition-proposition}[theorem]{Definition-Proposition}

\newtheorem{question}[theorem]{Question}

\theoremstyle{definition}
\newtheorem{definition}[theorem]{Definition}

\newtheorem{example}[theorem]{Example}

\newcommand{\Ext}{\operatorname{Ext}\nolimits}

\newcommand{\Hom}{\operatorname{Hom}\nolimits}
\newcommand{\End}{\operatorname{End}\nolimits}
\newcommand{\Tr}{\operatorname{Tr}\nolimits}
\renewcommand{\mod}{\mathsf{mod}\hspace{.01in}}

\newcommand{\rad}{\operatorname{rad}\nolimits}

\newcommand{\RHom}{\mathbf{R}\strut\kern-.2em\operatorname{Hom}\nolimits}

\numberwithin{equation}{section}

\usepackage{paralist}

\def\Im{\mathop{\rm Im}\nolimits}

\def\Coker{\mathop{\rm Coker}\nolimits}
\def\Tr{\mathop{\rm Tr}\nolimits}
\def\rad{\mathop{\rm rad}\nolimits}

\def\id{\mathop{\rm id}\nolimits}
\def\pd{\mathop{\rm pd}\nolimits}

\begin{document}
\title{A construction of Gorenstein projective $\tau$-tilting modules}
 \thanks{$*$ is the corresponding author. Both of the authors are supported by NSFC
(Nos. 11671174, 12171207). X. Zhang is supported by the Project Funded by the Priority
Academic Program Development of Jiangsu Higher Education Institutions and the Starting project of Jiangsu Normal University}
\author{Zhi-Wei Li}
\address{Z.W. Li:  School of Mathematics and Statistics, Jiangsu Normal University, Xuzhou, 221116, P. R. China.}
\email{zhiweili@jsnu.edu.cn}
\author{Xiaojin Zhang$^*$}
\address{X. Zhang:  School of Mathematics and Statistics, Jiangsu Normal University, Xuzhou, 221116, P. R. China.}
\email{xjzhang@jsnu.edu.cn, xjzhangmaths@163.com}
\maketitle
\begin{abstract} We give a construction of Gorenstein projective $\tau$-tilting modules in terms of tensor products of modules. As a consequence, we give a class of non-self-injective algebras admitting non-trivial Gorenstein projective $\tau$-tilting modules. Moreover, we show that a finite dimensional algebra $\Lambda$ over an algebraically closed field is $CM$-$\tau$-tilting finite if $T_n(\Lambda)$ is $CM$-$\tau$-tilting finite which gives a partial answer to a question on $CM$-$\tau$-tilting finite algebras posed by Xie and Zhang.
\end{abstract}

\section{Introduction}

 In 2014, Adachi, Iyama and Reiten \cite{AIR} introduced $\tau$-tilting theory as a generalization of tilting theory from the viewpoint of mutation. It has been showed by Adachi, Iyama and Reiten that $\tau$-tilting theory is closely related to silting theory and cluster tilting theory. In $\tau$-tilting  theory, (support) $\tau$-tilting modules are the most important objects. Therefore it is interesting to study (support) $\tau$-tilting modules for given algebras. For recent development on this topics, we refer to \cite{AiH,IZ,KK,W,Z,Zi}.

 On the other hand, Gorenstein projective modules form the main body of Gorenstein homological algebra; their origins can be traced back to Auslander-Bridger's modules of $G$-dimension zero \cite{AuB}. The definition of Gorenstein projective modules over an arbitrary ring was given by Enochs and Jenda \cite{EJ1,EJ2}. From then on, Gorenstein projective modules have gained a lot of attention in both homological algebra and the representation theory of finite-dimensional algebras. Throughout this paper, we focus on the finitely generated Gorenstein projective modules over finite dimensional algbras over an algebraically closed field $K$. For recent development of this topics, we refer to \cite{CSZ,HuLXZ,RZ1,RZ2,RZ3}.

 Recently, Xie and the second author \cite{XZ} combined Gorenstein projective modules with $\tau$-tilting modules and built a bijection map from Gorenstein projective support $\tau$-tilting modules to Gorenstein injective support $\tau^{-1}$-tilting modules which is analog to Adachi-Iyama-Reiten's bijection map from support $\tau$-tilting modules to support $\tau^{-1}$-tilting modules. But there is little reference to show the existence of non-trivial Gorenstein projective (support) $\tau$-tilting modules except support $\tau$-tilting modules over self-injective algebras. It is natural to ask the following question.
 \begin{question}\label{1.0} Are there non-self-injective algebras admitting non-trivial Gorenstein projective support $\tau$-tilting modules?
 \end{question}

 In this note, we give a construction of non-trivial Gorenstein projective (support) $\tau$-tilting modules. As a consequence, we can construct a large class of non-self-injective algebras admitting non-trivial Gorenstein projective (support) $\tau$-tilting modules. Our first main result is the following:

\begin{theorem}\label{1.1}$(\mathrm{Theorem}\  \ref{3.6})$ Let $A$ and $B$ be finite dimensional algebras over an algebraically closed field $K$. Let $M\in\mod B$ be a Gorenstein projective (support) $\tau$-tilting module. Then $A\otimes_K M$ is a Gorenstein projective (support) $\tau$-tilting module in $\mod (A\otimes_K B)$.
\end{theorem}

As a consequence of Theorem \ref{1.1}, we can give a partial answer to Question \ref{1.0} as follows.

\begin{corollary}\label{1.a} $(\mathrm{Corollary}\  \ref{3.d})$ Let $A$ be a non-semisimple self-injective algebra which is not local. Then there are non-trivial Gorenstein projective support $\tau$-tilting modules in $\mod T_n(A)$.
\end{corollary}

Recall from \cite{XZ} that an algebra is called Cohen-Macaulay-$\tau$-tilting finite ($CM$-$\tau$-tilting finite for short) if it admits finitely many isomorphism classes of indecomposable Gorenstein projective $\tau$-rigid modules.  The $CM$-$\tau$-tilting finite algebras are the generalizations of both algebras of finite Cohen-Macaulay type ($CM$-finite algebras for short) \cite{B, C,LZ1} and $\tau$-tilting finite algebras \cite{DIJ}. As an application of Theorem \ref{1.1}, we get the following characterization of $CM$-$\tau$-tilting finite algebras which gives a partial answer to \cite[Question 5.7]{XZ}.
\begin{theorem}\label{1.2}$(\mathrm {Theorem}\  \ref{3.7})$ Let $A$ be a finite dimensional algebra over an algebraically closed field $K$ and let $n\geq2$ be a positive integer. If $T_n(A)$ is $CM$-$\tau$-tilting finite, then $A$ is $CM$-$\tau$-tilting finite.
\end{theorem}

Now we show the organization of this paper as follows:

In Section 2, we recall some preliminaries on Gorenstein projective modules and $\tau$-tilting modules. In Section 3, we show the main results.

Throughout this paper, all algebras are finite-dimensional algebras over an algebraically closed field $K$ and all modules are finitely generated right modules.  We use $\tau$ to denote the Auslander-Reiten translation functor.

\bigskip
\noindent{\bf Acknowledgement} The authors would like to thank Profs. Xiao-Wu Chen, Zhaoyong Huang and Guodong Zhou for useful suggestions. The second author would like to thank Prof. Osamu Iyama for showing the fact a basic connected algebra with all $\tau$-rigid modules projective is local. The authors also want to thank the referee for his/her useful suggestions to improve the paper.

\section{Preliminaries}

In this section, we recall definitions and basic facts on $\tau$-tilting modules, tensor products of algebras and Gorenstein projective modules.

For an algebra $A$, denote by $\mod A$ the category of finitely generated right $A$-modules. We use $\mathcal{P}(A)$ to denote the subcategory of $\mod A$ consisting of projective modules. Now we recall the following definition of Gorenstein projective modules from \cite{EJ1}.

\begin{definition}\label{2.1}Let $A$ be an algebra and $M\in\mod A$. $M$ is called {\it Gorenstein projective}, if there is an exact sequence $\cdots \rightarrow P_{-1} \rightarrow P_{0} \rightarrow P_{1} \rightarrow \cdots$ in $\mathcal{P}(A)$, which stays exact under $\Hom_{A}(-,A)=(-)^{\ast}$, such that $M\simeq\Im ( P_{-1} \rightarrow P_{0})$.
\end{definition}

Denote by $\Omega$ the syzygy functor and $\Tr$ the Auslander-Bridger transpose functor.
The following properties of Gorenstein projective modules \cite{AuB} are essential.

\begin{proposition}\label{2.2} Let $A$ be be an algebra. A module $M\in\mod A$ is {\it Gorenstein projective} if and only if $M\simeq M^{\ast\ast}$ and $\Ext_{A}^{i}(M,A)=\Ext_{A}^{i}(M^{\ast},A)=0$ hold for all $i\geq 1$ if and only if $\Ext_{A}^{i}(M,A)=0$ and $\Ext_{A}^{i}(\Tr M,A)=0$ hold for all $i\geq 1$.
\end{proposition}





For a module $M\in \mod A$, denote by $|M|$ the number of pairwise
non-isomorphic indecomposable summands of $M$. We recall the definitions of $\tau$-rigid  modules and $\tau$-tilting modules from \cite{S} and \cite{AIR}.

\begin{definition}\label{2.4}Let $A$ be be an algebra and $M\in\mod A$.
\begin{enumerate}[\rm(1)]
\item We call $M$ {\it $\tau$-rigid} if $\Hom_{A}(M,\tau M)=0$. Moreover, $M$ is called a {\it $\tau$-tilting} module if $M$ is $\tau$-rigid
and $|M|=|A|$.
\item We call $M$ {\it support $\tau$-tilting} if there exists an idempotent $e$ of $A$ such that $M$ is a $\tau$-tilting $A/(e)$-module.
\end{enumerate}
\end{definition}

The following result \cite[Proposition 2.4(c)]{AIR} is essential in this paper.

\begin{proposition}\label{2.5}
 Let $M\in \mod A$ and $P_1(M)\stackrel{f}{\rightarrow}P_0(M)\rightarrow M\rightarrow 0$ be a minimal projective presentation of $M$.
 Then $M$ is $\tau$-rigid if and only if $\Hom_{A}(f,M)$ is epic.
 \end{proposition}

 We also need the following definitions of Gorenstein projective support $\tau$-tilting modules and Gorenstein projective $\tau$-tilting modules \cite{XZ}.

 \begin{definition}\label{2.a} Let $A$ be an algebra and $M\in\mod A$.
\begin{enumerate}[\rm(1)]
\item We call $M$ {\it Gorenstein projective $\tau$-rigid} if it is both $\tau$-rigid and Gorenstein projective.
\item We call $M$ {\it Gorenstein projective $\tau$-tilting} if it is both $\tau$-tilting and Gorenstein projective.
\item We call $M$ {\it Gorenstein projective support $\tau$-tilting} if it is both support $\tau$-tilting and Gorenstein projective.
\end{enumerate}
\end{definition}

Let $A$ and $B$ be algebras over an algebraically closed field $K$. Denote by $A\otimes_{K}B$ the tensor products of algebras. For modules
$M\in \mod A$ and $N\in\mod B$, we have $M\otimes_K N\in \mod A\otimes_K B$.
In the rest of the paper, we use $M\otimes N$ to denote $M\otimes_K N$. We need the following properties on the tensor products of algebras in \cite{CE}.

\begin{lemma}\label{2.6} Let $A$ and $B$ be two algebras with $M_i\in\mod A$ and $N_i\in\mod B$ for $i=1,2$. Then we have the following.
\begin{enumerate}[\rm(1)]
\item $\Hom_{A\otimes B}(M_1\otimes N_1,M_2\otimes N_2)\simeq \Hom_A(M_1,M_2)\otimes\Hom_A(N_1,N_2)$.
\item  $\Ext^{m}_{A\otimes B}(M_1\otimes N_1,M_2\otimes N_2)\simeq \bigoplus_{i+j=m}\Ext^i_A(M_1,M_2)\otimes\Ext^j_A(N_1,N_2)$ holds for $m\geq 1$.
\end{enumerate}
\end{lemma}

For a right $A$-module $M$, denote by $\pd_A M$ (resp. $\id_A M$) the projective (resp. injective) dimension of $M$. We also need the following on the injective (resp. projective) dimension of tensor products of modules.

\begin{lemma}\label{2.7} Let $A$ and $ B$ be two algebras with $M\in\mod A$ and $N\in \mod B$. Then
\begin{enumerate}[\rm(1)]
\item $\pd_{A\otimes B}M\otimes N=\pd_A M+\pd_B N$
\item $\id_{A\otimes B}M\otimes N=\id_A M+\id_B N$
\end{enumerate}
\end{lemma}
The following results are well-known.

\begin{proposition}\label{2.8} Let $A$ and $B$ be two algebras over an algebraically closed field $K$. Then
\begin{enumerate}[\rm(1)]
\item  $P\otimes Q$ is an indecomposable projective module in $\mod(A\otimes B)$ if $P$ and $Q$ are indecomposable projective in $\mod A$ and $\mod B$, respectively.
\item Every indecomposable projective module in $\mod(A\otimes B)$ has the form $P\otimes Q$, where $P$ and $Q$ are indecomposable projective in $\mod A$ and $\mod B$, respectively.
\item Every simple module in $\mod(A\otimes B)$ has the form $S\otimes S'$, where $S$ and $S'$ are simple modules over $A$ and $B$, respectively.
\end{enumerate}
\end{proposition}

\section{Main results}

In this section, we study the intersections among tensor products of algebras, $\tau$-rigid modules and Gorenstein projective modules. We give a method in constructing non-trivial Gorenstein projective support $\tau$-tilting modules. As a consequence, we can give a partial answer to the question posed by Xie and Zhang \cite[Question 5.7]{XZ}

The following properties on the indecomposable direct summands of tensor products of modules are essential in this paper.

\begin{proposition}\label{3.c} Let $A$ and $B$ be two algebras. If $M\in\mod A$ and $N\in\mod B$ are indecomposable modules, then $M\otimes N$ is an indecomposable module in $\mod(A\otimes B)$.
\end{proposition}

\begin{proof} By Lemma \ref{2.6}(1), there is an algebra isomorphism $$\End_{A\otimes B}(M\otimes N)\simeq \End_A(M)\otimes \End_B(N)$$ Denote by $I=\End_A(M)\otimes\rad(\End_B(N))+\rad(\End_A(M))\otimes \End_B(N)$, the radical of $\End_{A\otimes B}(M\otimes N)$. Then tensoring the short exact sequences $$0\rightarrow \rad(\End_A(M))\rightarrow \End_A(M)\rightarrow \End_A(M)/\rad(\End_A(M))\rightarrow 0$$ with $$0\rightarrow \rad(\End_B(N))\rightarrow \End_B(N)\rightarrow \End_B(N)/\rad(\End_B(N))\rightarrow 0$$ over $K$, one gets that $\End_A(M)\otimes \End_B(N)/I \simeq\End_A(M)/\rad(\End_A(M))\otimes (\End_B(N)/\rad(\End_B(N))$ is simple by Lemma \ref{2.8}(3). Then $\End_A(M)\otimes \End_B(N)$
is a local algebra. This implies that $M\otimes N$ is indecomposable.
\end{proof}

We have the following proposition.
\begin{proposition}\label{2.9} Let $A$ and $B$ be two algebras with $M\in \mod A$ and $N\in\mod B$.
\begin{enumerate}[\rm(1)]
\item $|M\otimes N|=|M||N|$ holds,
\item $|A\otimes B|=|A||B|$ holds.
\end{enumerate}
\end{proposition}
\begin{proof} It is easy to see that $M\otimes N_1\simeq M\otimes N_2$ in $\mod(A\otimes B)$ implies that $N_1\simeq N_2$ in $\mod B$. Then one gets the assertion by Proposition \ref{3.c}.
\end{proof}

Now we show the following proposition on tensor products of Gorenstein projective modules which is shown in \cite[Proposition 2.6]{HuLXZ}. We give a different proof in terms of functors.

\begin{proposition}\label{3.1} Let $A$ and $B$ be two algebras. Let $M\in \mod A$ and $N\in \mod B$ be Gorenstein projective modules. Then
$M\otimes N\in\mod(A\otimes B)$ is Gorenstein projective.
\end{proposition}

\begin{proof} Following Ringel and Zhang \cite{RZ1}, we call a module $M\in\mod A$ semi-Gorenstein projective if $\Ext_{A}^i(M,A)=0$ for all $i\geq 1$. By Proposition \ref{2.2}, we divide the proof into three steps.

(1) We show that $M\otimes N\in\mod(A\otimes B)$ is semi-Gorenstein projective if $M\in \mod A$ and $N\in \mod B$ are semi-Gorenstein projective.

Since $M$ and $N$ are both semi-Gorenstein projective, then $\Ext^i_A(M,A)=\Ext^j_B(N,B)=0$ holds for all $i\geq 1$ and $j\geq 1$. By Lemma \ref{2.6}(2) we get that $\Ext_{A\otimes B}^{m}(M\otimes N,A\otimes B)\simeq \bigoplus_{i+j=m}\Ext^i_A(M,A)\otimes\Ext^j_B(N,B)=0$ holds for $m\geq 1$.

(2) We show $(M\otimes N)^*$ is semi-Gorenstein projective if both $M^*$ and $N^*$ are semi-Gorenstein projective.

By Lemma \ref{2.6}(1), we get that $(M\otimes N)^*\simeq \Hom_{A\otimes B}(M\otimes N,A\otimes B)\simeq \Hom_A(M,A)\otimes\Hom_B(N,B)\simeq M^*\otimes N^*$. Then the assertion follows from $(1)$.

(3) We show that $M\otimes N$ is reflexive, that is, $M\otimes N\simeq(M\otimes N)^{**}$.

By (2)$(M\otimes N)^*\simeq M^*\otimes N^*$. Then one gets the assertion by using (2) once more.
\end{proof}

It has been shown in \cite{XZ} the quotient algebras of $CM$-finite algebras need not be $CM$-finite. However, we have the following result .

\begin{corollary}\label{3.d} Let $A$ be an algebra and let $T_n(A)$ be the lower triangular matrix algebra for $n\geq 2$. If $T_n(A)$ is $CM$-finite, then $A$ is $CM$-finite.
\end{corollary}

\begin{proof} Let $M_1, M_2$ be two indecomposable Gorenstein projective modules in $\mod A$ such that $T_n(K)\otimes M_1\simeq T_n(K)\otimes M_2\in \mod T_n(A)$. Then one gets $M_1\simeq M_2$ since $T_n(K)$ is projective over $K$.
\end{proof}

In the following we focus on the tensor products of $\tau$-rigid modules. In general, the tensor products of $\tau$-rigid modules need not be $\tau$-rigid. However, we have the following proposition.

\begin{proposition}\label{3.2} Let $A$ and $B$ be two algebras. Let $M\in\mod B$ be a $\tau$-rigid module and $P\in\mod A$ be a projective module.
Then $P\otimes M\in \mod(A\otimes B)$ is a $\tau$-rigid module.
\end{proposition}
\begin{proof} Let $P_1\stackrel{f}{\rightarrow} P_0\rightarrow M\rightarrow0$ be a minimal projective presentation of $M$. Then one gets the following minimal projective presentation of $P\otimes M$: $P\otimes P_1\stackrel{Id_P\otimes f}{\rightarrow} P\otimes P_0\rightarrow P\otimes M\rightarrow0$. By Proposition \ref{2.5}, it suffices to show that $\Hom_{A\otimes B}(Id_P\otimes f, P\otimes M): \Hom_{A\otimes B}(P\otimes P_0,P\otimes M)\rightarrow \Hom_{A\otimes B}(P\otimes P_1,P\otimes M) $ is a surjective map. By Lemma \ref{2.6}(1), the map above can be seen as: $\Hom_A(P,P)\otimes \Hom_B(P_0,M)\rightarrow \Hom_A(P,P)\otimes \Hom_B(P_1,M)$ via $g\otimes h\rightarrow gId_P\otimes hf$. Since $M$ is $\tau$-rigid, we get that $\Hom(f,M):\Hom_B(P_0,M)\rightarrow\Hom_B(P_1,M)$ is a surjective map. For any generator $k\otimes l\in \Hom_A(P,P)\otimes \Hom_B(P_1,M)$, we get a morphism $h$ such that $hf=l$. Therefore, $(k\otimes h)(Id_P\otimes f)=k\otimes l$ which implies the map $\Hom_{A\otimes B}(Id_P\otimes f, P\otimes M)$ is surjective. Then the assertion holds.
\end{proof}

The following proposition on tensor products of algebras is essential.

\begin{proposition}\label{3.3} Let $A$ and $B$ be two algebras. Let $(a)$ be a principal ideal of $A$ and $(b)$ be a principal ideal of $B$.
Then the principal ideal $(a\otimes b)=(a)\otimes (b)$.
\end{proposition}

\begin{proof} We first show $(a\otimes b)\subseteq(a)\otimes (b)$. For any element $m\in(a\otimes b)\subseteq A\otimes B$, one gets that $m=\Sigma_{i=1}^n (a_i\otimes b_i)(a\otimes b)(c_i\otimes d_i)=\Sigma_{i=1}^n a_iac_i\otimes b_ibd_i$. Since $a_iac_i\otimes b_ibd_i\in (a)\otimes (b)$ and $(a)\otimes (b)$ is an ideal of $A\otimes B$, we get that $m\in (a)\otimes (b)$.

Conversely, for any $n\in (a)\otimes (b)$, one gets that $n=\Sigma_{i=1}^{t}a_i\otimes b_i$, where $a_i=\Sigma_{k=1}^{t_i}a_{ik}ac_{ik}$ $b_i=\Sigma_{j=1}^{s_i}b_{ij}bd_{ij}$.
Thus $n=\Sigma_{i=1}^{t}\Sigma_{j=1}^{s_i}\Sigma_{k=1}^{t_i}a_{ik}ac_{ik}\otimes b_{ij}bd_{ij}$. Since $a_{ik}ac_{ik}\otimes b_{ij}bd_{ij}=(a_{ik}\otimes b_{ij})(a\otimes b)(c_{ik}\otimes d_{ij})$ is in $(a\otimes b)$ and $(a\otimes b)$ is an ideal, then the assertion holds.
\end{proof}

Now we are in a position to show our main result on support $\tau$-tilting modules.

\begin{theorem}\label{3.4} Let $A$ and $B$ be two algebras. Let $M\in\mod B$ be a support $\tau$-tilting module.
Then $A\otimes M\in \mod (A\otimes B)$ is a support $\tau$-tilting module.
\end{theorem}

\begin{proof} We divide the proof into two parts.

(1) We show that $A\otimes B/(e)\simeq (A\otimes B)/(1\otimes e)$, where $e$ is an idempotent of $B$.

Note that there is an exact sequence $0\rightarrow (e)\rightarrow B\rightarrow B/(e)\rightarrow 0$.
Applying the functor $A\otimes-$ to the exact sequence above, one gets the following exact sequence
$0\rightarrow A\otimes(e)\rightarrow A\otimes B\rightarrow A\otimes B/(e)\rightarrow 0$. By Proposition \ref{3.3}, one gets the assertion.

(2) We show that $A\otimes M$ is a $\tau$-tilting module over $(A\otimes B)/(1\otimes e)$.

Since $M$ is a support $\tau$-tilting module, then $M$ is a $\tau$-tilting module over $B/(e)$.
Then $|B/(e)|=|M|$. By Proposition \ref{3.2}, $A\otimes M$ is a $\tau$-rigid module. By Proposition \ref{2.9},
$|A\otimes M|=|A||M|=|A||B/(e)|=|(A\otimes B)/(1\otimes e)|$ by (1). The assertion holds.
\end{proof}

Now we have the following corollary on $\tau$-tilting modules.

\begin{corollary}\label{3.5} Let $A$ and $B$ be two algebras. Let $M\in\mod B$ be a $\tau$-tilting module.
Then $A\otimes M\in \mod(A\otimes B)$ is a $\tau$-tilting module.
\end{corollary}

\begin{proof} This is an immediate result of Theorem \ref{3.4}.
\end{proof}

Recall from \cite{DIJ} that an algebra $A$ is called $\tau$-tilting finite if it admits finite number of isomorphism classes of indecomposable $\tau$-rigid modules. We have the following corollary on $\tau$-tilting finite algebras.

\begin{corollary}\label{3.a} Let $A$ be an algebra and let $T_n(A)$ be the lower triangular matrix algebra. If $T_n(A)$ is $\tau$-tilting finite, then $A$ is $\tau$-tilting finite.
\end{corollary}

\begin{proof} This is clear since $A$ is a quotient algebra of $T_n(A)$.
\end{proof}

Now we are in a position to state the following main result on constructing Gorenstein projective support $\tau$-tilting modules.

\begin{theorem}\label{3.6} Let $A$ and $B$ be two algebras. If $M\in\mod B$ is a Gorenstein projective support $\tau$-tilting module,
then $A\otimes M\in \mod(A\otimes B)$ is a Gorenstein projective support $\tau$-tilting module.
\end{theorem}

\begin{proof} This is an immediate result of Proposition \ref{3.1} and Theorem \ref{3.4}.
\end{proof}

As a corollary, we get the following property on the existence of non-trivial Gorenstein projective support $\tau$-tilting modules.

\begin{corollary}\label{3.b} Let $A$ and $B$ be two algebras. If $M\in\mod B$ is a non-trivial Gorenstein projective support $\tau$-tilting module,
then $A\otimes M\in \mod(A\otimes B)$ is a non-trivial Gorenstein projective support $\tau$-tilting module.
\end{corollary}

\begin{proof} Since $M$ is non-trivial Gorenstein projective, then we get that $\pd_A M=\infty$ by Proposition \ref{2.2}. By Lemma \ref{2.7}, one gets that $\pd_A A\otimes M=\pd_B M=\infty$. Then the assertion follows from Theorem \ref{3.6}.
\end{proof}

Applying Corollary \ref{3.b} to self-injective algebras, we have the following result which gives a positive answer to Question \ref{1.0}.

\begin{corollary}\label{3.d} Let $A$ be a non-semisimple self-injective algebra which is not local. Then there are non-trivial Gorenstein projective support $\tau$-tilting modules in $\mod T_n(A)$.
\end{corollary}

\begin{proof} It is well-known that for an algebra $B$ if all $\tau$-rigid modules in $\mod B$ are projective, then $B$ is local. By the assumption, one gets that there are non-projective $\tau$-rigid modules in $\mod A$. And hence there are non-trivial support $\tau$-tilting modules $M\in\mod A$.
Then the assertion holds by Corollary \ref{3.b}.
\end{proof}

By using Corollary \ref{3.d}, one is able to construct a large class of non-trivial Gorenstein projective support $\tau$-tilting modules by taking a class of non-trivial Gorenstein projective support $\tau$-tilting modules over a self-injective algebra.

Recall from \cite{XZ} that an algebra $A$ is called $CM$-$\tau$-tilting finite if it admits a finite number of isomorphism classes of indecomposable $\tau$-rigid modules. It is an open question that whether $CM$-$\tau$-tilting finite algebras are closed under quotients. In the following we give a partial positive answer to the question.

\begin{theorem}\label{3.7} Let $A$ be an algebra and let $T_n(A)$ be the lower triangular matrix algebra for $n\geq 2$. If $T_n(A)$ is $CM$-$\tau$-tilting finite, then $A$ is $CM$-$\tau$-tilting finite.
\end{theorem}

\begin{proof} It is well-known that $T_n(A)\simeq T_n(K)\otimes A$. For any Gorenstein projective support $\tau$-tilting module $M\in \mod A$, by Theorem \ref{3.6}, one gets a Gorenstein projective support $\tau$-tilting module $T_n(K)\otimes M\in \mod T_n(A)$. Since $K$ is a field, we get $T_n(K)\otimes M \simeq T_n(K)\otimes N$ implies $M\simeq N \in\mod A$. Therefore, the fact that $T_n(A)$ is $CM$-$\tau$-tilting finite implies that $A$ is $CM$-$\tau$-tilting finite.
\end{proof}

Putting $n=2$, one gets the following result on the representation of Gorenstein projective $\tau$-tilting modules which combines the results in \cite[Corollary 4.7]{PMH} and \cite[Theorem 1.1]{LZ2}.

\begin{proposition} \label{3.8}  Let $A$ be a Gorenstein algebra and let $T_2(A)$ be the lower triangular matrix algebra. If
$M=\left(
\begin{smallmatrix}
X\\
Y\\
\end{smallmatrix}
\right)_{f}
$ is a Gorenstein projective (support) $\tau$-tilting module in $\mod T_2(A)$, then $Y$ is a Gorenstein projective (support) $\tau$-tilting module.
\end{proposition}

\begin{proof} It is shown in \cite{LZ2} that $M$ is Gorenstein projective if and only if both $X$ and $Y$ are Gorenstein projective, $f$ is a monomorphism and $\Coker f$ is Gorenstein projective. By \cite[Corollary 4.7]{PMH}, one gets that $Y$ should be a (support) $\tau$-tilting module. The assertion holds.
\end{proof}

For more details on support $\tau$-tilting modules over triangular matrix rings, we refer to \cite{GH,PMH}. For more details on tensor product algebras, we refer to \cite{L}.
We end this paper with the following example to show our main results.

\begin{example}\label{3.h} Let $A$ be the algebra given by the quiver $Q:\xymatrix{1\ar@<.2em>[r]^{a_{1}}&2\ar@<.2em>[l]^{a_{2}}}$ with the relations $a_1a_2=a_2a_1=0$. Let $n\geq 2$ be an integer. Then
\begin{enumerate}[\rm(1)]

\item $A$ is a self-injective algebra and $T_n(A)$ is not a self-injective algebra.

\item $ 1 \oplus \begin{smallmatrix} 1\\&2 \end{smallmatrix}$ and $1$ are two non-trivial Gorenstein projective support $\tau$-tilting modules in $\mod A$.

\item $T_n(K)\otimes( 1 \oplus \begin{smallmatrix} 1\\&2 \end{smallmatrix})$ and $T_n(K)\otimes 1$ are non-trivial Gorenstein projective support $\tau$-tilting modules in $\mod T_n(A)$.

\item $T_n(T_n(A))$ is also a non-self-injective algebra which admits non-trivial Gorenstein projective $\tau$-tilting modules.

\item By using the construction above, one can get infinite number of non-self-injective algebras admitting non-trivial Gorenstein projective $\tau$-tilting modules.

\end{enumerate}

\end{example}

\end{document}